\documentclass[a4paper,10pt]{amsart}
\usepackage{hyperref,fancyhdr,mathrsfs,amsmath,amscd,amsthm,amsfonts,latexsym,amssymb,stmaryrd,verbatim}
\usepackage[all]{xy}

\DeclareMathOperator{\trace}{trace}
\DeclareMathOperator{\grad}{grad}

\DeclareMathOperator{\dif}{d}

\newcommand{\V}{\mathscr{V}}

\newcommand{\ol}{\mathcal{O}}

\def \o{\omega}

\def \phi{\varphi}
\def \Phi{\varPhi}
\def \p{\pi}

\def \R{\mathbb{R}}
\def \Hq{\mathbb{H}\,}
\def \C{\mathbb{C}\,}

\def\widecheckg{g^{\hspace*{-2.5pt}\vbox to 5pt{\hbox to
0pt{\LARGE$\check{}$}}}\hspace*{2pt}}

\def\widecheckl{\lambda^{\hspace*{-3.5pt}\vbox to 8pt{\hbox to
0pt{\LARGE$\check{}$}}}\hspace*{2pt}}

\newcommand{\sn}{{\mathbb S}}
\newcommand{\rn}{{\mathbb R}}

\newcommand{\cn}{{\mathbb C}}
\renewcommand{\i}{{\rm i}}
\newcommand{\cp}{{\mathbb C}P^n}
\newcommand{\tcp}{T{\mathbb C}P^n}
\newcommand{\SU}{\mathrm{SU}(n+1)}
\newcommand{\U}{\mathrm{U}}
\newcommand{\HM}{\mathrm{HM}(n+1)}
\newcommand{\su}{\mathfrak{su}(n+1)}

\DeclareMathOperator{\proj}{proj}

\begin{document}

\title{Harmonic morphisms and moment maps\\ 
on hyper-K\"ahler manifolds} 
\author{M.~Benyounes, E.~Loubeau, R.~Pantilie} 
\address{M.~Benyounes, D\'epartement de Math\'ematiques, Universit\'e de Bretagne Occidentale, 6,
avenue Victor Le Gorgeu, CS 93837, 29238 Brest Cedex 3, France} 
\email{\href{mailto:Michele.Benyounes@univ-brest.fr}{Michele.Benyounes@univ-brest.fr}}
\address{E.~Loubeau, D\'epartement de Math\'ematiques, Universit\'e de Bretagne Occidentale, 6,
avenue Victor Le Gorgeu, CS 93837, 29238 Brest Cedex 3, France} 
\email{\href{mailto:Eric.Loubeau@univ-brest.fr}{Eric.Loubeau@univ-brest.fr}} 
\address{R.~Pantilie, Institutul de Matematic\u a ``Simion~Stoilow'' al Academiei Rom\^ane,
C.P. 1-764, 014700, Bucure\c sti, Rom\^ania} 
\email{\href{mailto:Radu.Pantilie@imar.ro}{Radu.Pantilie@imar.ro}}

\newtheorem{thm}{Theorem}[section]
\newtheorem{lem}[thm]{Lemma}
\newtheorem{cor}[thm]{Corollary}
\newtheorem{prop}[thm]{Proposition}

\theoremstyle{definition}

\newtheorem{defn}[thm]{Definition}
\newtheorem{rem}[thm]{Remark}
\newtheorem{exm}[thm]{Example}

\numberwithin{equation}{section}

\begin{abstract} 
We characterise the actions, by holomorphic isometries on a K\"ahler manifold with zero first Betti number, of an abelian Lie group of $\dim\geq2$\,, 
for which the moment map is horizontally weakly conformal (with respect to some Euclidean structure on the Lie algebra of the group).\\ 
\indent 
Furthermore, we study the hyper-K\"ahler moment map $\phi$ induced by an abelian Lie group $T$ acting by triholomorphic isometries on a hyper-K\"ahler manifold $M$,  
with zero first Betti number, thus obtaining the following:\\ 
\indent 
$\bullet$ If $\dim T=1$ then $\phi$ is a harmonic morphism. Moreover, we illustrate this on the tangent bundle of the complex projective space equipped with the 
Calabi hyper-K\"ahler structure, and we obtain an explicit global formula for the map.\\ 
\indent 
$\bullet$ If $\dim T\geq2$ and either $\phi$ has critical points, or $M$ is nonflat and $\dim M=4\dim T$ then $\phi$ cannot be horizontally weakly conformal. 
\end{abstract} 
\maketitle
\thispagestyle{empty}
\vspace{-10mm}

\section{Introduction} 

This article is a contribution towards establishing the following statement: \emph{any canonical submersion/projection 
between Riemannian manifolds is, in a natural way, a (generalized) harmonic morphism}; that is, it preserves, through the pull back, the sheaves of harmonic functions (or natural 
subsheaves of it). 

Devised to produce new harmonic functions on Riemannian manifolds, harmonic morphisms quickly revealed a very strong geometric nature as horizontally weakly conformal harmonic maps, i.e.\ at the same time a critical point of the energy functional and, away from a critical set, a submersion acting conformally on the horizontal distribution 
(see \cite{BaiWoo2}\,). 

These combined conditions give an over-determined system, which needs not admit solutions and, indeed, no general existence result has yet emerged. There exist, nonetheless, classification theorems in some instances such as maps from three-dimensional space forms (see \cite{BaiWoo2}\,) or with one-dimensional fibres on Einstein manifolds 
\cite{Pan-4to3}\,,\,\cite{PanWoo-d}\,.

A common thread in all these examples is the isometric action of a Lie group on the domain and such favourable conditions can be encountered on a wider scale~\cite{isom}\,.

This theme can be taken to an even higher level on K\"ahler manifolds, since the symplectic structure allows the construction of a moment map which, under conditions, gives a new type of harmonic morphisms.

While, strictly speaking, this scheme fails to produce interesting examples in K\"ahler geometry, essentially for dimensional reasons (see, also, Theorem \ref{thm:hcmom_Kahler}\,, 
and Corollary \ref{cor:2.5}\,, below), its extension to hyper-K\"ahler manifolds turns out to be more successful. The existence of a trio of K\"ahlerian structures and, potentially, of triholomorphic isometric actions, allows for moment maps which are harmonic morphisms in the three-dimensional Euclidean space (Corollary \ref{cor:hm_to_R3}\,). 
Basically, this constructs three harmonic functions with (mutually) orthogonal gradients. 

For higher dimensional abelian Lie group actions, the corresponding hyper-K\"ahler moment maps are still harmonic 
and twistorial (Propositions \ref{prop:hK_har_mom} and \ref{prop:twist_hK_mom}\,, respectively; 
see \cite{LouPan-II} for more information on twistorial maps). However, at least in the presence of critical points, 
or assuming nonflatness and a natural dimensional condition, these cannot be horizontally weakly conformal 
(Corollary \ref{cor:hK_fixed_points_hmom_n=1}\,, Theorem \ref{thm:3.7}\,, 
respectively; see, also, Example \ref{exm:2.2}\,), with respect to the metric induced by some Euclidean structure on the corresponding Lie algebra. 
Nevertheless (see the proof of Proposition \ref{prop:twist_hK_mom}\,), the twistoriality is equivalent to the fact that the hyper-K\"ahler moment map, 
of any abelian Lie group $T$ acting by triholomorphic isometries, pulls back, into the sheaf of harmonic functions, a natural subsheaf of the sheaf of harmonic functions on 
$(\mathfrak{t}\otimes{\rm Im}\Hq)^*$\,, where $\mathfrak{t}$ is the Lie algebra of $T$. 

In dimension four, a harmonic morphism is a (local) hyper-K\"ahler moment map if and only if it is given by the Gibbons--Hawking construction. 
This leads to natural classes of hyper-K\"ahler moment maps, for all possible dimensions (Example \ref{exm:Gibbons-Hawking}\,). 

We, also, implement this approach on the tangent bundle of the complex projective space equipped with Calabi's hyper-K\"ahler metric, 
with explicit descriptions of spaces and maps. Working with a natural complex homogeneous model for $\cp$, we give detailed formulas for the constituents of the hyper-K\"ahler structure (see Section \ref{section:Calabi}\,) and show how Killing fields on $\cp$ give rise to Killing fields on $\tcp$, with tri-holomorphic isometric actions. Integrating the three Hamiltonian functions provides an explicit globally defined harmonic morphism from $\tcp$ to $\rn^3$, parametrized by $\su$ (Theorem \ref{thm:hmom_Calabi}\,). 
Dimension one must stand out, as for $n=1$ we obtain a harmonic morphism with one-dimensional fibres and this map must fall into one of the previously 
described classifications: it is produced, as a projection, by the foliation defined by the distinguished Killing field on $T\C\!P^1$, in a manner reminiscent of the Hopf map.

\section{Horizontally conformal moment maps on K\"ahler manifolds} 

\indent 
All the manifolds are assumed connected. Also, for simplicity, we restrict ourselves to abelian Lie group actions, 
although most of our results hold in the more general setting of generalized foliations 
generated by a finite family of commuting (tri)holomorphic isometries.\\  
\indent 
An orbit of a group action is called \emph{nontrivial} if it is not a fixed point. All the actions are assumed nontrivial 
(that is, there exists at least one nontrivial orbit). 

\begin{prop} \label{prop:hcmom_Kahler} 
Let $T$ be an abelian Lie group acting by holomorphic isometries on a K\"ahler manifold $(M,g,J)$ with zero first Betti number. 
Denote by $\phi:M\to\mathfrak{t}^*$ the moment map, where $\mathfrak{t}$ is the Lie algebra of $T$.\\ 
\indent 
The following assertions are equivalent:\\ 
\indent 
\quad{\rm (i)} The moment map of $T$ is horizontally weakly conformal, with respect to some Euclidean metric on $\mathfrak{t}$\,;\\ 
\indent 
\quad{\rm (ii)} There exists an invariant metric on $T$ such that the induced maps from $T$ onto the nontrivial orbits are homothetic covering maps;\\ 
\indent 
\quad{\rm (iii)}  The nontrivial orbits of $T$ are umbilical and have the same dimension as $T$. 
\end{prop} 
\begin{proof} 
(i)$\Longleftrightarrow$(ii)\;  If $v\in\mathfrak{t}$ we shall denote by $V$ the corresponding (fundamental) vector field on $M$.\\ 
\indent 
Let $h$ be a Euclidean metric on $\mathfrak{t}$ and let $(v_i)$ be a basis of $\mathfrak{t}$\,. Then (i) is satisfied, with respect to $h$\,, if and only if 
(see \cite[Lemma 2.4.4]{BaiWoo2}\,)  there exists a nonnegative function $\Lambda$ on $M$ such that $g(\dif\!\phi_i,\dif\!\phi_j)=\Lambda h_{ij}$\,, 
for any $i$ and $j$, where $h_{ij}=h(v_i,v_j)$ and $\phi_i=v_i\circ\phi$\,.\\ 
\indent 
As $\dif\!\phi_i$ is the one-form which, under the musical isomorphisms, corresponds to $JV_i$\,, we have that (i) holds, with respect to $h$\,, 
if and only if there exists a nonnegative function $\Lambda$ such that $g(V_i,V_j)=\Lambda h_{ij}$\,, for any $i$ and $j$\,; 
in particular, as $T$ is acting by isometries, $\Lambda$ is constant along the orbits. The proof of this equivalence quickly follows.\\  
\indent 
(ii)$\Longleftrightarrow$(iii)\; A foliation on a Riemannian manifold is umbilical if and only if the local 
flows of the basic vector fields, restricted to each leaf, are conformal diffeomorphisms. Therefore this equivalence follows from the fact that 
the metrics induced on the orbits are invariant.
\end{proof} 

\indent 
Here is an application of Proposition \ref{prop:hcmom_Kahler}\,. 

\begin{exm}  \label{exm:2.2} 
Let $T\subseteq{\rm SU}(n+1)$ be an abelian Lie subgroup acting canonically on $\C\!P^n$ (endowed with the Fubini-Study metric). From Proposition \ref{prop:hcmom_Kahler} 
we obtain that the corresponding moment map is horizontally weakly conformal if and only if $\dim T=1$\,.\\ 
\indent 
Indeed, if $\dim T\geq2$ then, by the Gauss' equation, the sectional curvature of a plane tangent to a nontrivial orbit would be nonpositive which, 
together with \cite[(11.2.5)]{BaiWoo2} applied to the Hopf projection from $S^{2n+1}$ to $\C\!P^n$, leads to a contradiction.  
\end{exm} 
 
\indent 
The following fact should be well-known but we do not have a reference for it. 

\begin{lem} \label{lem:isotropic_orbits} 
Let $T$ be an abelian Lie group acting by symplectic diffeomorphisms on a symplectic manifold, with zero first Betti number. 
Denote by $\phi:M\to\mathfrak{t}^*$ the moment map, where $\mathfrak{t}$ is the Lie algebra of $T$ and by $\V$ the generalized 
foliation formed by the orbits.\\ 
\indent 
Then ${\rm ker}\dif\!\phi$ is the orthogonal complement of $\V$ with respect to the symplectic form; in particular, 
the orbits are isotropic. 
\end{lem} 
\begin{proof} 
If $v\in\mathfrak{t}$ we shall denote by $V$ the induced vector field on $M$. By definition, the function $x\mapsto\phi(x)(v)$ 
is the Hamiltonian of $V$; we shall denote it by $\phi_v$\,.\\ 
\indent 
We have that $\dif\!\phi(X)=0$\,, for some $X\in TM$, if and only if $\dif\!\phi_v(X)=0$\,, for any $v\in\mathfrak{t}$\,; 
equivalently, $\o(V,X)=0$\,, for any $v\in\mathfrak{t}$\,, where $\o$ is the symplectic form.\\ 
\indent 
The fact that $\V$ is isotropic follows from the fact that $\V\subseteq{\rm ker}\dif\!\phi$\,. 
\end{proof}

\indent 
In the following theorem $T^{\C}\!$ is any complexification of the abelian Lie group $T$. 

\begin{thm} \label{thm:hcmom_Kahler} 
Let $T$ be an abelian Lie group, $\dim T\geq2$, acting by holomorphic isometries on a K\"ahler manifold $(M,g,J)$ with zero first Betti number. 
Denote by $\phi:M\to\mathfrak{t}^*$ the moment map, where $\mathfrak{t}$ is the Lie algebra of $T$.\\ 
\indent 
Then the following assertions are equivalent:\\ 
\indent 
\quad{\rm (i)} $\phi$ is horizontally weakly conformal, with respect to some Euclidean metric on~$\mathfrak{t}$\,;\\ 
\indent 
\quad{\rm (ii)} The action of $T$ extends to a holomorphic local action of $T^{\C\!}$ whose orbits are K\"ahler submanifolds which, 
locally and up to homotheties, can be identified with $T^{\C\!}$ endowed with its canonical complex structure and an invariant K\"ahler metric.\\
\indent 
Moreover, if {\rm (i)} or {\rm (ii)} holds then $\phi$ is horizontally homothetic. 
\end{thm} 
\begin{proof} 
Denote by $\V$ be the generalized foliation formed by the orbits. {}From Lemma \ref{lem:isotropic_orbits}\,, 
we obtain that $J\V$ is the distribution orthogonal to the fibres of $\phi$\,. Thus, if (ii) holds then $\phi$ is horizontally homothetic.\\ 
\indent 
To prove (i)$\Longrightarrow$(ii)\,, we may restrict to the open subset of $M$ where $\phi$ is a submersion.\\ 
\indent 
As $T$ is abelian and acts by holomorphic diffeomorphisms, $J\V$ and $\V\oplus J\V$ are integrable and $T$ extends to a holomorphic 
local action of $T^{\C}$ whose orbits are the leaves of $\V\oplus J\V$.\\ 
\indent 
To complete the proof, it is sufficient to consider the case $(M,J)=T^{\C\!}$.\\ 
\indent 
Let $(v_1,\ldots,v_n)$ be an orthonormal basis of $\mathfrak{t}$\,. Then $(V_1,JV_1,\ldots,V_n,JV_n)$ is a commuting frame on $T^{\C\!}$ 
which, by Proposition \ref{prop:hcmom_Kahler}\,, is conformal with respect to $g$\,; in particular, in an open neighbourhood of each point of 
$T^{\C\!}$ there exists coordinates $(s_1,t_1,\ldots,s_n,t_n)$ such that $\partial/\partial s_i=V_i$ and $\partial/\partial t_i=JV_i$\,, for any $i=1,\ldots,n$\,. 
Then the K\"ahler form is given by $\o=\Lambda\sum_i^n\dif\!s_i\wedge\dif\!t_i$\,, for some positive function $\Lambda$\,. As $\dif\!\o=0$\,, 
$\Lambda$ must be constant which completes the proof. 
\end{proof}   

\indent 
The following result is an immediate consequence of Theorem \ref{thm:hcmom_Kahler}\,. 

\begin{cor} \label{cor:2.5} 
Let $T$ be an abelian Lie group acting by holomorphic isometries on a K\"ahler manifold $M$ with zero first Betti number, 
$\dim M=2\dim T\geq4$\,. Denote by $\phi:M\to\mathfrak{t}^*$ the moment map, where $\mathfrak{t}$ is the Lie algebra of $T$.\\ 
\indent 
If $\phi$ is horizontally weakly conformal, with respect to some Euclidean metric on~$\mathfrak{t}$\,, then $M$ is flat. 
\end{cor}  

\indent 
In all the results of this section we may, obviously, replace the assumption on the first Betti number with the (less restrictive) 
assumption that the action is Hamiltonian. This, also, applies to the results of the next section.

\section{Harmonic morphisms and moment maps on hyper-K\"ahler manifolds} 

\indent 
Let $(M,g,J_1,J_2,J_3)$ be a hyper-K\"ahler manifold, with zero first Betti number. A K\"ahler structure $J$ on $(M,g)$ is called \emph{admissible} 
(for the given hyper-K\"ahler structure) if there exists $q\in S^2(\subseteq{\rm Im}\Hq$\!) such that $J=q^jJ_j$\,. 

\begin{prop} \label{prop:hK_har_mom} 
Let $T$ be an abelian Lie group acting by triholomorphic isometries on a hyper-K\"ahler manifold $M$ with zero first Betti number.\\ 
\indent 
Then the induced moment map, with respect to (the K\"ahler form of) any admissible K\"ahler structure on $M$, is a harmonic map. 
\end{prop} 
\begin{proof} 
It is sufficient to prove that the Hamiltonian of any triholomorphic Killing vector field, with respect to an admissible K\"ahler structure, 
is a harmonic function.\\ 
\indent 
Let $(g,J_1,J_2,J_3)$ be the hyper-K\"ahler structure of $M$, and let $V$ be a triholomorphic Killing vector field; 
denote by $\o_1$\,, $\o_2$\,, $\o_3$ the K\"ahler forms of $J_1$, $J_2$\,, $J_3$\,, respectively.\\ 
\indent 
Then $\o_2+{\rm i}\,\o_3$ is a holomorphic symplectic form on $(M,J_1)$ which is preserved by $V$ and, in particular, by $V^{1,0}$. 
Consequently, if $f_2$ and $f_3$ are the Hamiltonians of $V$, with respect to $\o_2$ and $\o_3$\,, respectively, then $f_2+{\rm i}f_3$ is 
the (holomorphic) Hamiltonian of $V^{1,0}$ with respect to $\o_2+{\rm i}\,\o_3$\,; in particular, $f_2+{\rm i}f_3$ is a harmonic function 
on~$(M,g)$\,. 
\end{proof}  

\begin{cor} \label{cor:hm_to_R3}
Let $V$ be a triholomorphic Killing vector field on a hyper-K\"ahler manifold $(M,g,J_1,J_2,J_3)$\,, with zero first Betti number.\\ 
\indent 
Then $(f_1,f_2,f_3)$ is a harmonic morphism from $(M,g)$ to $\R^3$, where $f_j$ is the Hamiltonian of $V$ with respect to the K\"ahler form 
of $J_j$\,, $j=1,2,3$\,. 
\end{cor}   

\begin{exm} 
The canonical Poisson structure on the dual of a Lie algebra $\mathfrak{g}$ is characterised by the fact that, for any $A\in\mathfrak{g}$\,, 
the corresponding linear function on $\mathfrak{g}^*$ is the Hamiltonian of the vector field determined by the transpose of ${\rm ad}A$\,; 
in particular, the corresponding symplectic leaves are the co-adjoint orbits.\\ 
\indent 
If $\mathfrak{g}$ is the Lie algebra of a compact semisimple Lie group then (see \cite{Biq} and the references therein) 
the complex symplectic structure on any co-adjoint orbit $M$, on the dual of $\mathfrak{g}^{\C}$, 
underlies a $G$-invariant hyper-K\"ahler structure.\\ 
\indent  
Let $J_1$ be the complex structure on $M$, induced from $\mathfrak{g}^{\C}$, and let $J_2$ and $J_3$ be admissible K\"ahler structures  
such that $(J_1,J_2,J_3)$ satisfy the quaternionic identities.\\ 
\indent
Then, any $A\in\mathfrak{g}$\,, defines a complex valued harmonic morphism on $M$, which gives the last two components of the 
harmonic morphism of Corollary \ref{cor:hm_to_R3}\,, determined by the transpose of ${\rm ad}A$\,.\\ 
\indent 
In contrast, the first component of this harmonic morphism (that is, the Hamiltonian with respect to $J_1$ of the vector field on $M$ 
determined by ${\rm ad}A$\,) does not seem to admit such a simple and explicit description (compare \cite[Lemma 1]{D'Amo_S-C}\,). 
Nevertheless, in Theorem \ref{thm:hmom_Calabi}\,, below, we shall give an explicit description for a significant particular case. 
\end{exm} 

\indent 
Let $T$ be an abelian Lie group acting by triholomorphic isometries on a hyper-K\"ahler manifold $M$ with zero first Betti number; 
denote by $\mathfrak{t}$ the Lie algebra of $T$.\\ 
\indent 
Let $\phi:M\to(\mathfrak{t}\otimes{\rm Im}\Hq)^*$ be the map characterised by the fact that, for any 
$v\in\mathfrak{t}$ and unit quaternion $q\in{\rm Im}\Hq$, the function $x\mapsto\phi(x)(v\otimes q)$ is the Hamiltonian 
of the (fundamental) vector field on $M$ corresponding to $v$ with respect to the admissible K\"ahler structure corresponding to $q$\,.\\ 
\indent 
The map $\phi$ is called the \emph{hyper-K\"ahler moment map}. Note that the codomain of $\phi$ is a co-CR quaternionic vector space 
whose twistor space is $n\ol(2)$\,, where $n=\dim T$ and $\ol(2)$ is the holomorphic line bundle, over the Riemann sphere, of Chern number~$2$~\cite{fq_2}.  

\begin{exm} \label{exm:Gibbons-Hawking} 
Let $a>0$\,. Then  $(\R^4,g_a)$ is a (self-dual) nonflat complete hyper-K\"ahler manifold 
(see \cite{Pan-4to3} and the references therein), where 
\begin{equation*}
g_a=(a|x|^2+1)\,g_0-\frac{a(a|x|^2+2)}{a|x|^2+1}(-x_2\dif\!x_1+x_1\dif\!x_2-x_4\dif\!x_3+x_3\dif\!x_4)^2\;, 
\end{equation*} 
with $g_0$ the canonical Euclidean metric on $\R^4$.\\ 
\indent 
Furthermore, $\phi:\bigl(\R^4,g_a\bigr)\to\R^3$\,, 
$\bigl(z_1,z_2\bigr)\mapsto\bigl(|z_1|^2-|z_2|^2, 2z_1\overline{z_2}\bigr)$\,, 
is a harmonic morphism given by the Gibbons-Hawking construction. It follows that $\phi$ is the hyper-K\"ahler moment map 
of the action of $S^1$ on $(\R^4,g_a)$ given by $\zeta\cdot(z_1,z_2)=(\zeta z_1,\zeta z_2)$\,.\\  
\indent 
By taking direct products we easily obtain hyper-K\"ahler moment maps for all possible dimensions of the abelian Lie group 
and the hyper-K\"ahler manifold. Similarly, we obtain harmonic morphisms as in Corollary \ref{cor:hm_to_R3} for all 
possible dimensions of the domain. 
\end{exm}

\begin{prop} \label{prop:twist_hK_mom} 
Let $T$ be an abelian Lie group acting by triholomorphic isometries on a hyper-K\"ahler manifold $M$ with zero first Betti number. 
Denote by $\phi:M\to(\mathfrak{t}\otimes{\rm Im}\Hq)^*$ the hyper-K\"ahler moment map, where $\mathfrak{t}$ is the Lie algebra 
of $T$.\\ 
\indent 
Then $\phi$ is a twistorial map; equivalently, $\phi$ corresponds to a holomorphic map from the twistor space of $M$ to $n\ol(2)$\,,  
mapping the twistor lines onto the images of sections of $n\ol(2)$\,, where $n=\dim T$. 
\end{prop} 
\begin{proof} 
From Lemma \ref{lem:isotropic_orbits}\,, we obtain that the fixed point set of the action is equal to the zero set $S$ of $\dif\!\phi$\,. 
Hence, $M\setminus S$ is a nonempty open set (as we work only with nontrivial actions).\\ 
\indent 
To prove that $\phi$ is twistorial outside $S$, from \cite[Proposition 3.1]{fq} it follows that it is sufficient to prove the 
following fact, for each admissible K\"ahler structure $J$ on $M$\,: \emph{the differential of $\phi$ maps ${\rm ker}(J+\sqrt{-1})$\,, 
into the (constant) distribution given by $\mathfrak{t}^{\C}\otimes p^{\perp}$, where $p\in({\rm Im}\Hq)^{\C}$ (depends only of $J$ and) 
is isotropic with respect to the complexification of the canonical Euclidean structure of ${\rm Im}\Hq$}.\\ 
\indent 
We claim that this property holds with $p=q_1+\sqrt{-1}\,q_2$\,, where $(q,q_1,q_2)$ is a positive orthonormal frame on ${\rm Im}\Hq$ 
with $q$ the unit quaternion corresponding to $J$. Indeed, let $\psi_J:\mathfrak{t}^*\otimes{\rm Im}\Hq\to\mathfrak{t}\otimes\C$ be 
the linear projection with kernel $\mathfrak{t}^*\otimes\R q$\,, where $\C\subseteq{\rm Im}\Hq$ is generated by $q_1$ and $q_2$\,. 
On identifying ${\rm Im}\Hq=({\rm Im}\Hq)^*$, through its Euclidean structure, we may compose $\phi$ followed by $\psi_J$ 
thus obtaining a holomorphic map from $(M\setminus S,J)$ to $\mathfrak{t}^{\C}$. 
The fact that $\psi_J\circ\phi$ is holomorphic is equivalent to the claim.\\ 
\indent 
Now, obviously, $\psi_J\circ\phi$ extends to a well-defined holomorphic map on $(M,J)$\,; in particular, the zero set of its differential has empty interior. 
Consequently, also, the zero set of the differential of $\phi$ has empty interior which is sufficient to complete the proof. 
\end{proof}  

\begin{cor} \label{cor:hK_fixed_points_hmom_n=1} 
Let $T$ be an abelian Lie group acting by triholomorphic isometries on a hyper-K\"ahler manifold $M$ with zero first Betti number. 
Denote by $\phi:M\to(\mathfrak{t}\otimes{\rm Im}\Hq)^*$ the hyper-K\"ahler moment map, where $\mathfrak{t}$ 
is the Lie algebra of $T$.\\ 
\indent 
If the action has a fixed point then $\phi$ is horizontally weakly conformal, with respect to some Euclidean metric on $\mathfrak{t}$\,, 
if and only if $\dim T=1$\,. 
\end{cor} 
\begin{proof} 
Suppose, towards a contradiction, that $\dim T\geq2$ and $\phi$ is horizontally weakly conformal, with respect to some Euclidean metric on $\mathfrak{t}$\,. 
Then, by Theorem \ref{thm:hcmom_Kahler} and Proposition \ref{prop:hK_har_mom}\,, $\phi$ is a horizontally homothetic harmonic morphism.  
Consequently, it is a submersion (see \cite[Corollary 4.5.5]{BaiWoo2}\,); that is, the fixed point set is empty, a contradiction. 
\end{proof} 

\indent 
By \cite{isom}\,, any foliation, of ${\rm codim}\neq2$\,, generated by an isometric action of an abelian Lie group is locally defined by harmonic morphisms. 
In the current setting, more can be said: 

\begin{thm} \label{thm:3.7} 
Let $T$ be an abelian Lie group acting by triholomorphic isometries on a hyper-K\"ahler manifold $M$ with zero first Betti number, 
$\dim M=4\dim T\geq8$\,. Let $\phi:M\to(\mathfrak{t}\otimes{\rm Im}\Hq)^*$ be the hyper-K\"ahler moment map, 
where $\mathfrak{t}$ is the Lie algebra of~$T$.\\ 
\indent  
If $\phi$ is horizontally weakly conformal, with respect to some Euclidean metric on~$\mathfrak{t}$\,, then $M$ is flat. 
\end{thm}
\begin{proof} 
By the proof of Corollary \ref{cor:hK_fixed_points_hmom_n=1}\,, $\phi$ is a horizontally homothetic submersion. Furthermore, as 
$\dim M=4\dim T$, the orbits are connected components of the fibres of $\phi$\,. Thus, together with Proposition \ref{prop:hcmom_Kahler}\,, 
we obtain that, up to a homothety, $\phi$ is a Riemannian submersion. Moreover, Propositions \ref{prop:hcmom_Kahler} and \ref{prop:hK_har_mom} 
imply that the fibres of $\phi$ are flat and geodesic.\\ 
\indent 
To complete the proof it is sufficient to prove that the horizontal distribution (that is, the orthogonal complement of ${\rm ker}\dif\!\phi$\,) is integrable.\\ 
\indent 
If $v\in\mathfrak{t}$ we shall denote by $V$ the induced vector field on $M$, and if $q$ is an imaginary unit quaternion we shall 
denote by $J_q$ the corresponding admissible K\"ahler structure. Note that, for any such $v$ and $q$\,, the musical dual of $J_qV$ 
is the pullback by $\phi$ of the differential of the linear function on $(\mathfrak{t}\otimes{\rm Im}\Hq)^*$ defined by $v\otimes q$\,. 
Equivalently, on identifying $(\mathfrak{t}\otimes{\rm Im}\Hq)^*=\mathfrak{t}\otimes{\rm Im}\Hq$ through the isomorphism 
induced by the Euclidean metrics, we have that $J_qV$ is the horizontal lift of the (constant) vector field on $\mathfrak{t}\otimes{\rm Im}\Hq$ 
defined by $v\otimes q$\,.\\ 
\indent 
Consequently, for any imaginary unit quaternions $p$ and $q$ and any vector fields $U$ and $V$ induced by elements 
of $\mathfrak{t}$\,, we have that $[J_pU,J_qV]$ is a vertical vector field.\\ 
\indent 
Let $p$\,, $q$ and $r$ be orthonormal imaginary quaternions and let $U$, $V$ and $W$ be vector fields induced by elements of $\mathfrak{t}$\,. 
On denoting by $\o_p$ the K\"ahler form of $J_p$\,, we obtain $0=\dif\!\o_p(J_pU,J_qV,J_rW)=-g\bigl(U,[J_qV,J_rW]\bigr)$\,, 
thus completing the proof. 
\end{proof}

\section{The Calabi metric on the tangent bundle\\ 
of the complex projective space} \label{section:Calabi} 

\indent 
Let $\HM = \bigl{\{}A\in \mathfrak{gl}(n+1,\cn)\, |\, \bar{A}=A^t\bigr{\}}$ be the set of Hermitian $(n+1)\times (n+1)$-matrices. $\HM$ is an $(n+1)^2$-dimensional 
complex vector space and we equip it with the Hermitian metric $(A,B)\mapsto 2\trace (AB) ,$ for $A$ and $B$ in $\HM$.\\
\indent 
Then (see \cite{Ros}\,) $\cp$ can be embedded into $\HM$ as the subset 
$$\bigl{\{} A\in \HM \,|\, A^2=A\,,\,\trace A =1 \bigr{\}}\;.$$ 
Let $\U(n)$ be the unitary group, then $\cp$ is a submanifold of $\HM$ diffeomorphic to $\U(n+1)/\bigl(\U(1)\times\U(n)\bigr)$\,. 
The manifold $\cp$ can actually be seen as the orbit of the adjoint action of $\U(n+1)$ on $\HM$\,, at 
$$A_0 = \begin{pmatrix}
1 &0 & 0\\
0 & 0 &0\\
0 & 0 & 0 
\end{pmatrix} 
\,.$$
\indent 
With respect to the metric on $\HM$\,, the $\U(n+1)$ action is isometric and the induced metric on $\cp$ is isometric with the Fubini--Study metric, with holomorphic sectional curvature equal to $1$\,. Note that, the complex structure $J$ on $\cp$ is given by 
$JX = {\rm i}(I-2A)X$, for all $A\in \cp$ and $X\in T_A \cp$.\\ 
\indent 
Later on, we shall need the following relations which can be checked directly at the point $A_0$ and then extended to the whole of 
$\cp$ by the $\SU$ action (beware that the first one is only valid for $n=1$):
\begin{equation} \label{formules} 
\begin{split} 
XY+YX &= \trace(XY) I, \quad  X,Y \in T_A {\mathbb C}P^1\;;\\
(XY+YX)A &= \trace(XY) A, \quad  X,Y \in T_A {\mathbb C}P^n\;;\\
2(XXY + YXX + 2XYX) &= \trace(XX) Y, \quad  X,Y \in T_A {\mathbb C}P^n, X\perp Y . 
\end{split} 
\end{equation} 
\indent 
Let $\p:\tcp\to \cp$ be the projection. As $T\cp$ is a vector bundle we have a canonical embedding of $\p^*\bigl(\tcp\bigr)$ into 
${\rm ker}\dif\!\p$\,. Accordingly, any $Y\in T_A\cp$\,, for some $A\in\cp$\,, has a \emph{vertical} lift $Y^v$ which is a section 
of $T\bigl(T_A\cp\bigr)\,\bigl(\subseteq T\bigl(\tcp\bigr)\,\bigr)$\,.\\ 
\indent 
Under the embedding of $\cp$ into $\HM$\,, by using the fact that taking vertical lifts is canonical, we deduce that, 
for any $X\in T_A\cp$ we have $Y^v = (0,Y)\in T_{(A,X)}\tcp$.\\ 
\indent 
On the other hand, the Levi--Civita connection of $\cp$ determines a splitting of the exact sequence 
$0\longrightarrow\p^*\bigl(\tcp\bigr)\longrightarrow T\bigl(\tcp\bigr)\longrightarrow\p^*\bigl(\tcp\bigr)\longrightarrow0\,,$  
where the projection from $T\bigl(\tcp\bigr)$ onto $\p^*\bigl(\tcp\bigr)$ is induced by $\dif\!\p$.\\ 
\indent 
Therefore any $Y\in T_A\cp$\,, for some $A\in\cp$\,, also, has a \emph{horizontal} lift $Y^h\in T_{(A,X)}\tcp$\,, for any
$X\in T_A\cp$\,. A straightforward argument shows that $Y^h=\bigl(Y,\,{\rm i}\,[X,JY]\bigr)\in T_{(A,X)}\tcp$.\\ 
\indent 
We can now describe \emph{Calabi's hyper-K\"ahler metric} $G$ on $\tcp$ (compare \cite{BG}\,). At each $(A,X)\in\tcp$ it is 
given by 
\begin{align*}
G(U^h,V^h) &= \tfrac{a+1}{2} g(U,V) + \tfrac{a-1}{2|X|^2}\left[g(U,X)g(V,X) + g(U,JX)g(V,JX)\right]\;;\\
G(U^h,V^v) &= 0\;;\\
G(U^v,V^v) &= \tfrac{2}{a+1} g(U,V) - \tfrac{a-1}{2a(a+1)|X|^2}\left[g(U,X)g(V,X) + g(U,JX)g(V,JX)\right]\;,
\end{align*}
where $a= \sqrt{1+4|X|^2}$, $U,V\in T_A\cp$ and $g$ is the Fubini-Study metric on $\cp$ at the point $A$.\\ 
\indent 
On ${\mathbb C}P^1$ the formulas greatly simplify into
\begin{align*}
G(U^h,V^h) &= a\,g(U,V) ;\\
G(U^h,V^v) &= 0 ;\\
G(U^v,V^v) &= \tfrac{1}{a}\,g(U,V) .
\end{align*}
\indent 
The three complex structures on $\tcp$ are given, at each $(A,X)\in\tcp$\,, as follows:\\ 
\indent
\quad(1) the canonical complex structure $J^*$ on $T^{*}\cp\,\bigl(=\tcp\bigr)$\,:
\begin{align*}
J^* (U^h) &= (JU)^h\;,\\
J^* (U^v) &= - (JU)^v\;;
\end{align*}
\indent
\quad(2) the complex structure $I^*$:
\begin{align*}
I^* (U^h) &= \tfrac{a+1}{2} U^v + \tfrac{a-1}{2|X|^2}\left[g(U,X)X + g(U,JX)JX\right]^v\;,\\
I^* (U^v) &= - \tfrac{2}{a+1}U^h + \tfrac{a-1}{2a(a+1)|X|^2}\left[g(U,X)X + g(U,JX)JX\right]^h\;;
\end{align*} 
\indent 
\quad (3) the product $K^*= I^* \circ J^*$:
\begin{align*}
K^* (U^h) &= \tfrac{a+1}{2} (JU)^v + \tfrac{a-1}{2|X|^2}\left[g(U,X)JX - g(U,JX)X\right]^v\;,\\
K^* (U^v) &=  \tfrac{2}{a+1}(JU)^h - \tfrac{a-1}{2a(a+1)|X|^2}\left[g(U,X)JX - g(U,JX)X\right]^h\;.
\end{align*}
\indent
The Levi--Civita connection $\bar{\nabla}$ of $G$ is given by
\begin{align*}
\bar{\nabla}_{U^h} V^h &= \left(\nabla_{U}V\right)^h - \tfrac{1}{2} \left[ R(U,V)X + R(U,JY)JX\right]^v\;,\\
\bar{\nabla}_{U^v} V^h &= \tfrac{1}{2} \left( \mathcal{A}(U,V)\right)^h\;,\\
\bar{\nabla}_{U^h} V^v &= \left(\nabla_{U}V\right)^v + \tfrac{1}{2} \left( \mathcal{A}(V,U)\right)^h\;,\\
\bar{\nabla}_{U^v} V^v &= \left( \mathcal{B}(U,V)\right)^v\;,
\end{align*}
where $\nabla$ and $R$ are the Levi--Civita connection and the curvature tensor of $\cp$\,, respectively, 
and the tensors $\mathcal{A}$ and $\mathcal{B}$ are defined by
\begin{align*}
\mathcal{A}(U,V) &= \tilde{a} \left[ g(U,X)V - g(U,JX)JV\right]\\
&+ \tilde{a} \left[ g(U,V) - \tilde{a}\left[ g(U,X)g(V,X) + g(U,JX)g(V,JX)\right]\right] X\\
&+ \tilde{a} \left[ g(JU,V) - \tilde{a}\left[ g(U,X)g(V,JX) - g(U,JX)g(V,X)\right]\right] JX
\end{align*}
and 
\begin{align*}
\mathcal{B}(U,V) &= - \tfrac{1}{2}\tilde{a} \left[ g(V,X)U + g(U,X)V + g(V,JX)JU +g(U,JX)JV\right]\\
&+ \tfrac{1}{2}\tilde{a}^2 \left[ g(U,X)g(V,X) - g(U,JX)g(V,JX)\right] X\\
&+ \tfrac{1}{2}\tilde{a}^2 \left[ g(U,X)g(V,JX) + g(U,JX)g(V,X)\right] JX
\end{align*}
where $\tilde{a} = \tfrac{a-1}{a |X|^2}$ and $\mathcal{B}$ is bilinear, symmetric and satisfies
$$ \mathcal{B}(JU,V)=\mathcal{B}(U,JV)= J \mathcal{B}(U,V).$$
\indent 
One can check that $(G,I^*,J^*,K^*)$ is a hyper-K\"ahler structure on $\tcp$. Moreover, the 
canonical lift to $\tcp$ of the action of ${\rm SU}(n+1)$ on $\cp$ is isometric and triholomorphic.

\section{Harmonic morphisms from $\tcp$} 

\indent 
In this section, we prove the following result. 

\begin{thm} \label{thm:hmom_Calabi} 
If $u\in\su$\,, the map $(f_1,f_2,f_3): \tcp \to \rn^3$ is a harmonic morphism, where 
$\tcp$ is endowed with the Calabi metric and 
\begin{equation*} 
\begin{split} 
f_1(A,X)&= g({\rm i}u,JX)\;,\\ 
f_2(A,X)&= a\, g(A,{\rm i}u) - \tfrac{2}{a+1}\,g(X^2,{\rm i}u)\;,\\ 
f_3(A,X)&= g({\rm i}u,X)\;,   
\end{split} 
\end{equation*} 
with $a= \sqrt{1+4|X|^2}$, for any $(A,X)\in\tcp$,  and $g$ the Fubini--Study metric. 
\end{thm} 
\begin{proof} 
Let $\gamma$ be the Killing field on $\cp$, defined by $u\in \su$, and $\Gamma$ the induced triholomorphic Killing field on $\tcp$\,. 
By Corollary \ref{cor:hm_to_R3}\,, it is sufficient to prove that $(f_1,f_2,f_3)$ is the hyper-K\"ahler moment map determined by $\Gamma$\,.\\ 
\indent 
Firstly, note that 
\begin{align*}
\Gamma : \tcp &\to \tcp \\
(A,X) &\mapsto (\gamma(A),\gamma(X))= (\gamma(A))^h + (\nabla_{X}\gamma)^v \in T_{(A,X)}\tcp .
\end{align*}
Then 
\begin{align*}
I^* \Gamma(A,X) &= \tfrac{a+1}{2} (\gamma(A))^v + \tfrac{a-1}{2|X|^2}\left[g(\gamma(A),X)X + g(\gamma(A),JX)JX\right]^v\\
& - \tfrac{2}{a+1} (\nabla_{X}\gamma)^h + \tfrac{a-1}{2a(a+1)|X|^2}g(\nabla_{X}\gamma,JX)(JX)^h\;; \\
J^*\Gamma (A,X) &= (J\gamma(A))^h + (-\nabla_{X}\gamma)^v\;; \\
K^* \Gamma(A,X) &= \tfrac{a+1}{2} (J\gamma(A))^v + \tfrac{a-1}{2|X|^2}\left[g(\gamma(A),X)JX - g(\gamma(A),JX)X\right]^v\\
&+ \tfrac{2}{a+1}(J\nabla_{X}\gamma)^h + \tfrac{a-1}{2a(a+1)|X|^2}g(\nabla_{X}\gamma,JX)X^h\;. 
\end{align*}

Assuming $X\neq 0$, we construct the orthonormal basis $\{ E_1 = \tfrac{X}{|X|}, E_2=\tfrac{JX}{|X|}, E_i\}_{i=3,\dots,n}$ of $T_A\cp$ and take its horizontal and vertical lifts, with norms
\begin{align*}
&|E_1^h|^2= |E_2^h|^2= a, |E_i^h|^2= \tfrac{a+1}{2}, \quad  i \geq 3\,; \\
&|E_1^v|^2= |E_2^v|^2= \tfrac{1}{a}, |E_i^v|^2= \tfrac{2}{a+1}, \quad  i \geq 3\,, 
\end{align*}
and recall the formulas for the derivation of a function on the tangent bundle, at the point $(A,X)$:
\begin{align} \label{eqtan}
Y^v (f\circ \pi)&= Y(f)\circ\pi, \quad Y^v (f\circ\pi)=0 ,\\
Y^h (f(|X|^2)) &=0 , \quad Y^v (f(|X|^2)= 2 f'(|X|^2)g(X,Y) ,\notag\\
Y^h (g(Z,X)\circ\pi) &= g(\nabla_Y Z,X)\circ\pi . \notag
\end{align}
\indent 
\quad1) As $f_1(A,X)= g(\i u,JX)$, we have, since $J\gamma = \proj_{T_A\cp}(\i u)$
\begin{align*}
E_i^h (f_1) &= E_i^h (g(\i u,JX)) = E_{i}^h (g(\gamma,X))  \\
&= g(\nabla_{E_i} \gamma ,X)= -g(\nabla_{X}\gamma, E_i) ,
\end{align*}
so
\begin{align*}
\sum_{i=1}^{n} E_i^h (f_1)\frac{E_i^h}{|E_i^h|^2} &= \tfrac{1}{a} \left( -g\left(\nabla_{X}\gamma, \tfrac{X}{|X|}\right)\tfrac{X^h}{|X|} 
-g\left(\nabla_{X}\gamma, \tfrac{JX}{|X|}\right)\tfrac{(JX)^h}{|X|}\right) \\
&- \sum_{i=3}^{n} \tfrac{2}{a+1} g(\nabla_{X}\gamma, E_i)E_i^h\\
&= \tfrac{2}{a+1} (-\nabla_{X}\gamma)^h + \tfrac{a-1}{a(a+1)|X|^2}\left(  -g(\nabla_{X}\gamma, X)X^h -g(\nabla_{X}\gamma, JX)(JX)^h\right)
\end{align*}
which is exactly the horizontal part of $I^* \Gamma$.\\
For the vertical part,
\begin{align*}
E_i^v (f_1) &= E_i^v (g(\i u,JX)) = E_i^v (g(\gamma,X)) \\
&= (0,E_i) (g(\gamma,X)) = g(\gamma,E_i) 
\end{align*}
and
\begin{align*}
\sum_{i=1}^{n} E_i^v (f_1) &= a \left[ g\left(\gamma,\tfrac{X}{|X|}\right)\tfrac{X^v}{|X|} + g\left(\gamma,\tfrac{JX}{|X|}\right)\tfrac{(JX)^v}{|X|}\right] 
+ \tfrac{a+1}{2} \sum_{i=3}^{n} g(\gamma,E_i) E_i^v \\
&= \tfrac{a+1}{2} \gamma^v + \tfrac{a-1}{2|X|^2}\left[ g(\gamma,X)X^v + g(\gamma,JX)(JX)^v\right],
\end{align*}
which is the vertical part of $I^*\Gamma$.\\ 
\indent 
\quad2) As $f_3(A,X)= g(\i u,X)=g(J\gamma,X)$, we have  
\begin{align*}
E_i^h (f_3) &= E_i^h (g(J\gamma,X)) = g(\nabla_{E_i} J\gamma ,X) = g(\nabla_{X}J\gamma, E_i)\;, 
\end{align*}
so
\begin{align*}
\sum_{i=1}^{n} E_i^h (f_3)\frac{E_i^h}{|E_i^h|^2} &= \tfrac{1}{a} \left[ g\left(\nabla_{X}J\gamma, \tfrac{X}{|X|}\right)\tfrac{X^h}{|X|} 
+g\left(\nabla_{X}J\gamma, \tfrac{JX}{|X|}\right)\tfrac{(JX)^h}{|X|}\right] \\
&+ \sum_{i=3}^{n} \tfrac{2}{a+1} g(\nabla_{X}J\gamma, E_i)E_i^h\\
&= \tfrac{2}{a+1} (\nabla_{X}J\gamma)^h - \tfrac{a-1}{a(a+1)|X|^2}g(\nabla_{X}J\gamma, X)X^h ,
\end{align*}
which is exactly the horizontal part of $K^* \Gamma$.\\
For the vertical part,
\begin{align*}
E_i^v (f_3) &= E_i^v (g(\i u,X)) = E_i^v (g(J\gamma,X)) \\
&= (0,E_i) (g(J\gamma,X)) = g(J\gamma,E_i) ,
\end{align*}
and
\begin{align*}
\sum_{i=1}^{n} E_i^v (f_3) &= a \left[ g\left(J\gamma,\tfrac{X}{|X|}\right)\tfrac{X^v}{|X|} + g\left(J\gamma,\tfrac{JX}{|X|}\right)\tfrac{(JX)^v}{|X|}\right] 
+ \tfrac{a+1}{2} \sum_{i=3}^{n} g(J\gamma,E_i) E_i^v \\
&= \tfrac{a+1}{2} (J\gamma)^v + \tfrac{a-1}{2|X|^2}\left[ g(J\gamma,X)X^v + g(\gamma,X)(JX)^v\right]
\end{align*}
which is the vertical part of $K^*\Gamma$.\\
\indent 
\quad3) Lastly, as $f_2(A,X) = a g(A,\i u) - \tfrac{2}{a+1}g(XX,\i u)$ and $J^* \Gamma(A,X) = (J\gamma(A))^h + (-J\nabla_{X} \gamma(A))^v$, 
we have 
$$\grad^{\tcp} f_2 = \sum_{i=1}^n E_i^h(f_2)\frac{E_i^h}{|E_i^h|^2} + E_i^v(f_2)\frac{E_i^v}{|E_i^v|^2}.$$
Since $X^h = (X, \i [X,JX]) = (X, 2(XAX-AXX))$,
\begin{align*}
E_1^h(f_2) &= \tfrac{1}{|X|} X^h (f_2) = \tfrac{1}{|X|} \Big\{ a X^h(g(A,\i u)) - \tfrac{2}{a+1} X^h(g(XX,\i u))\Big\} \\
&= \tfrac{1}{|X|} \Big\{a g(A,\i u) - \tfrac{2}{a+1}g(XT+TX,\i u)\Big\} ,
\end{align*}
where $T= 2(XAX-AXX)$ and therefore $XT+TX=0$, so
$$E_1^h(f_2) = \tfrac{a}{|X|} g(A,\i u).$$
Since $(JX)^h= (JX,0)$,
\begin{align*}
E_2^h(f_2) &= \tfrac{1}{|X|} (JX)^h (f_2) =\tfrac{1}{|X|} \Big\{ a (JX)^h(g(A,\i u)) - \tfrac{2}{a+1} (JX)^h(g(XX,\i u))\Big\} \\
&= \tfrac{a}{|X|}g(JX,\i u).
\end{align*}
For $i\geq 3$, $(E_i)^h = (E_i , \i [X,JE_i])$ but
$$\i [X,JE_i] = E_i X + X E_i -2 (E_i X + XE_i )A = E_i X + X E_i ,$$
because $(E_i X + XE_i)A= \trace(E_i X)A =0$ (by \eqref{formules}), so $(E_i)^h = (E_i , E_i X + X E_i)= (E_i, T_i)$.
Therefore
\begin{align*}
E_i^h(f_2) &= a E_i^h(g(A,\i u)) - \tfrac{2}{a+1} E_i^h(g(XX,\i u)) \\
&= a g(E_i,\i u) - \tfrac{2}{a+1} g( T_i X + XT_i, \i u)
\end{align*}
but 
$$ T_i X + XT_i = (E_i X + X E_i)X + X(E_i X + X E_i) = E_i XX + XXE_i -2 XE_i X =\Big(\tfrac{\trace X^2}{2}\Big) E_i,$$
by Equation~\eqref{formules}, and 
\begin{align*}
E_i^h(f_2) &= a g(E_i,\i u) - \tfrac{2}{a+1} \Big(\tfrac{\trace X^2}{2}\Big) g( E_i , \i u) \\
&= \Big( a - \tfrac{\trace X^2}{a+1}\Big)g( E_i , \i u). 
\end{align*}

For the vertical parts,
\begin{align*}
E_1^v(f_2) &= \tfrac{1}{|X|} X^v (f_2) = \tfrac{1}{|X|}\Big\{ \tfrac{4}{a}|X|^2 g(A,\i u) + \tfrac{2c}{a(a+1)}|X|^2 g(XX,\i u) - \tfrac{2}{a+1} X^v(g(XX,\i u))\Big\}\\
&= \tfrac{1}{|X|}\Big\{ \tfrac{4}{a}|X|^2 g(A,\i u) - \tfrac{2}{a} g(XX,\i u)\Big\} ,
\end{align*}
since $X^v(g(XX,\i u)) = X(g(XX,\i u)) =2g(XX,\i u)$. However
\begin{align*}
E_2^v(f_2) &= \tfrac{1}{|X|} (JX)^v (f_2) = \tfrac{1}{|X|}\Big\{ - \tfrac{2}{a+1}g(X(JX) + (JX)X,\i u)\Big\}\\
&= 0 ,
\end{align*}
by \eqref{eqtan} and \eqref{formules}.\\
Finally, for similar reasons, if $i\geq 3$
\begin{align*}
E_i^v(f_2) &=  - \tfrac{2}{a+1}g(XE_i + E_iX,\i u),
\end{align*}
hence
\begin{align*}
\grad^{\tcp} f_2 &= \tfrac{1}{|X|^2}\Big\{ g(X,\i u) X^h + g(JX,\i u) (JX)^h\Big\} + \tfrac{2}{a+1}\Big( a - \tfrac{\trace X^2}{a+1}\Big)g( E_i , \i u) E_i^h \\
&+ \Big\{ 4 g(A,\i u) - \tfrac{2}{|X|^2} g(XX,\i u)\Big\} X^v - g(XE_i + E_iX,\i u) E_i^v .
\end{align*}

To establish equality with $J^*\Gamma$, we compute
\begin{align*}
-J\nabla_{X} \gamma &= g(-J\nabla_{X} \gamma,X) \frac{X}{|X|^2} + g(-J\nabla_{X} \gamma,JX) \frac{JX}{|X|^2}+ g(-J\nabla_{X} \gamma,E_i) E_i ,
\end{align*}
and
\begin{align*}
g(-J\nabla_{X} \gamma,E_i) &= g(\nabla_{X} \gamma,JE_i) = g(D_{X} \gamma,JE_i) = g(\gamma(X),JE_i)\\
&=g([u,X],JE_i) = g([X,JE_i],u)
\end{align*}
with
\begin{align*}
[X,JE_i] &= \i [X,[E_i ,A]] = \i \{ XE_i A -XAE_i + AE_i X -E_iAX\}\\
&= \i \{(E_i X + XE_i)A - XAE_i - E_iAX\} = \i \{ 2(E_i X + XE_i)A - (XE_i + E_iX)\}\\
&= -\i (XE_i + E_iX)
\end{align*}
since $(E_i X + XE_i)A = \trace(XE_i)A =0$ by \eqref{formules}. Therefore
$$g(-J\nabla_{X} \gamma,E_i) = g(-(XE_i + E_iX),\i u).$$
On the other hand
$$g(-J\nabla_{X} \gamma ,JX)=-g(\nabla_{X} \gamma ,X)=0,$$
and
\begin{align*}
g(-J\nabla_{X} \gamma ,X)&= g(\nabla_{X} \gamma ,JX)=g( \gamma(X) ,JX)=g([u,X],JX)\\
&= g([X,JX],u)= 2\i g(2XXA-XX,u)=2\i g((\trace XX)A- XX, u)\\
&= 2(\trace XX)g(A,\i u) -2g(XX,\i u) \quad \text{(by Formula~\eqref{formules}).}
\end{align*}
Therefore
\begin{align*}
\Big(-J\nabla_{X} \gamma\Big)^v &= \frac{1}{|X|^2} \Big\{ 2(\trace XX)g(A,\i u) -2g(XX,\i u)\Big\} X^v 
-g(XE_i + E_iX,\i u) E_i^v
\end{align*}
For the horizontal part
\begin{align*}
J\gamma(A) &= \i (I-2A)[u,A] = \i \Big( uA - Au -2A(uA-Au)\Big)\\
&= \i \Big( uA + Au -2 AuA\Big) = (\i u)^{T_A\cp} ,
\end{align*}
so
$$J\gamma(A) = g (\i u,X) \frac{X^h}{|X|^2} + g (\i u,JX) \frac{(JX)^h}{|X|^2} +g (\i u,E_i) E_i^h,$$
which is indeed equal to the horizontal part of $\grad^{\tcp}f_2$ since 
$$\tfrac{2}{a+1}\left(a - \tfrac{\trace XX}{a+1}\right)=1.$$
\end{proof} 

\begin{rem} 
It is known that on a $2n$-dimensional compact K\"ahler-Einstein manifold with positive scalar curvature $s$, a vector field is Killing if and only it is equal to $J$ times the gradient field of an eigenfunction of the Laplacian for the eigenvalue $s/n$\,. While $\tcp$ is not compact this property persists for Killing fields which naturally come from Killing fields on $\cp$. Contrastingly, the Killing field $(A,X)\mapsto (JX)^v$ satisfies this property for $J^*$ but not for~$I^*$.
\end{rem} 

\begin{rem}
On $S^2 = {\mathbb C}P^1$, the expressions for the functions simplify to
\begin{align*}
f_1 (p,e) &= g(b,Je);\\
f_2 (p,e) &= \sqrt{1+|e|^2} g(b,p);\\
f_3 (p,e) &= g(b,e),
\end{align*}
where $(p,e)\in T\sn^2$ and the Killing field $\gamma$ on $\sn^2$ is defined by rotations around $b\in \rn^3$.\\ 
\indent 
This gives a new description for \cite[Example 4.7]{Pan-4to3}\,, in the case when 
the starting harmonic function has exactly two singular points.  
\end{rem}

\end{document}